\documentclass[12pt,reqno]{amsart}
\usepackage{amsfonts,amsmath,amssymb}
\usepackage[all]{xy}
\newtheorem{thm}{Theorem}[section]
\newtheorem{proposition}[thm]{Proposition}
\newtheorem{corollary}[thm]{Corollary}
\newtheorem{lemma}[thm]{Lemma}
\newtheorem{definition}[thm]{Definition}
\newtheorem{example}[thm]{Example}
\newtheorem{remark}[thm]{Remark}

\usepackage{graphicx}
\usepackage{epstopdf}

\title {Euler characters for  general linear  Lie superalgebra}

\author{A.N. Sergeev}\address{Department of Mathematics, Saratov State University, Astrakhanskaya 83, Saratov 410012   Russian Federation.}
\address{Department of Mathematics,    Saratov State University,\ \
Astrakhanskaya 83, Saratov 410012  and Moscow Institute of Physics and Technology,
9 Institutskiy per., Dolgoprudny, Moscow Region, 141701, Russin Federation}

\email{SergeevAN@info.sgu.ru,\,sergeev.an@mipt.ru}

\begin{document}

\maketitle

\begin{abstract}  M. Gorelik and Th. Heidersdors investigated Euler characters for Lie superalgebra  $\frak{gl}(m,n)$ and $\frak{osp}(m.2n)$. In the present paper we also  investigate  Euler characters for Lie superalgebra $\frak{gl}(m,n)$ but we use  a different approach and our  results  are formulated in different terms.
 \end{abstract}
 
\tableofcontents
\section{Introduction}  Let $K(\frak{gl}(m,n))$ be the ring of  the characters finite dimensional representations of the Lie superalgebra $\frak{gl}(m,n)$. There are two   natural bases in this ring:  the first consisting of the characters  of irreducible representations and the second consisting of the characters of  Euler modules.  The main goal of the  present paper  is an explicit formulae for the decomposition character of finite dimensional irreducible module in terms of  the characters of Euler modules. In the paper \cite{GH} authors investigated the decomposition  of irreducible modules in terms of Euler modules for all classical simple Lie superalgebras, but their technique and final formulae are different from our ones. 

Now let us formulate our  main result. Let $\mathcal P(m,n)$ be the set of the weight diagrams such that $m(f)=m, n(f)=n$ (see Definition \ref {weight}) and by $\mathcal E(m,n)$ we will denote the set of the weight diagrams $g$  such that  $m(g)\le m, n(g)\le n$ and $ n(g)-m(g)=n-m$. The irreducible characters  for  $\frak{gl}(m,n)$ can be enumerated by the elements of  $\mathcal P(m,n)$   and Euler    characters can be enumerated by the elements from  $\mathcal E(m,n)$.  Then we prove the following  formula
$$
ch\,L(g)=\sum_{p}(-1)^{\theta(f)+\theta(p)}e_{g,p} ch\,E(p)
$$

where $g\in \mathcal P(m,n), p\in \mathcal E(m,n)$
$$
 e_{g,p}=\sum_{q*Z=p}(-1)^{|Z|}d^{\,\infty}_{g,q},\quad Z\subset \Bbb Z_{\le n-m}
$$ 
and $d^{\infty}_{g,a}$ are given by simple explicit formula in Lemma \ref{infty}.
 
If  $g^{-1}(\times)\subset \Bbb Z_{>n-m}$ then the above formula has especially simple form 
$$
ch\,L(g)=\sum_{p}(-1)^{\theta(f)+\theta(p)}d_{g,p}^{\infty}ch\,E(p).
$$

 As a corollary of this decomposition  and   the stability of Euler characters under Duflo-Serganova homomorphism we  get the formula for image of the irreducible module under Duflo-Serganova homomorphism and formula for super-dimension. For the proofs of our results we need two main fact: decomposition  of ireduciible module  and Euler character in terms of Kac modules. For  reader convenience we  also give a new proof of Su-Zhang formula by means of creations operators.  
 \section{Preliminares}

Let us remind that the Lie superalgebra $\frak{gl}(m,n)$ is the Lie  superalgebra of the  linear transformations of a $\Bbb Z_2$ graded vector space $V=V_0\oplus V_1$ ($V$ is also called the standard representation of $\frak g$). We have 
$$
\frak g_{\bar0}=\frak{gl}(m)\oplus\frak{gl}(n),\quad\frak g_{\bar1}=V_0\otimes V_1^*\oplus V_1\otimes V_0^*.
$$ 
We also have $\Bbb Z$ graded decomposition 
$
\frak g=\frak g_{-1}\oplus \frak g_0\oplus \frak g_1
$
where
$$
\frak g_{-1}=V_1\otimes V^*_0,\,\, \frak g_{1}=V_0\otimes V^*_1\,.
$$
Let us fix  bases in $V_0=<e_1,\dots,e_m>$ and $V_1=<f_1,\dots,f_n>$ respectively. 
Let  $\frak{b}$ be the subalgebra of upper triangular matrix in $\frak{gl}(m,n)$ and $\frak{k}$ be  the    subalgebra of diagonal matrix in  $\frak{gl}(m,n)$ in the above basis. By   $\varepsilon_1,\dots,\varepsilon_m,\delta_1,\dots,\delta_n$  we will denote  the weights of standard representation with respect  to $\frak k$. The corresponding system of positive roots  $R^+=R^+_0\cup R^+_1$  of $\frak{gl}(m,n)$ can  be described in the following way
$$
R^+_{0}=\{\varepsilon_i-\varepsilon_j\,:\, 1\le i< j\le m\,:\, \delta_k-\delta_l,\, 1\le k<l\le n \}
$$
$$
R_1^+=\{\varepsilon_i-\delta_k,\, 1\le i\le m,\,1\le k\le n\}.
$$

 Let also  
$$
P=\{\chi=\lambda_1\varepsilon_1+\dots+\lambda_m\varepsilon_m+\mu_1\delta_1+\dots+\mu_n\delta_n,\mid n_i,m_j\in \Bbb Z\}
$$
be the weight lattice and 
$$
P^+=\{\chi\in P\mid \lambda_i-\lambda_j\ge0,\,i<j\,:\mu_k-\mu_l\ge0,\,k<l\}
$$
be the set of highest weights.

We will use the following parity on the weight lattice due to C. Gruson  and V. Serganova \cite{GS}  and Brundan and Stroppel \cite{BS} by saying  that $\varepsilon_i$ (resp. $\delta_j$) is even (resp. odd). It is easy to check that every finite dimensional module $L$ can be represented in the form
$$
L=L^+\oplus L^{-}
$$
where $L^+$ is the submodule of $L$ in which weight space has the same parity as the corresponding weight and $L^{-}$ is the submodule in which  the parities differ.  We should note that this construction is a particular case of Deligne construction category $Rep(G,z)$ from the paper \cite{D}  for $G=GL(m,n)$ and $z=diag(\underbrace{1,\dots,1}_{m},\underbrace{-1,\dots,-1}_{n}).$

Let us denote by $\mathcal F$ the category of finite dimensional modules over $\frak{gl}(n,m)$ such that every module in $\mathcal F$ is semisimple over Cartan  subalgebra $\frak k$ and and all  its weights are in $P$.
By  $K(\mathcal F )$ we will denote the quotient of the Grothendieck ring of   $\mathcal F$ by the relation $[L]-[\Pi(L)]=0$ where $\Pi(L)$ is the module  with the shifted parity $\Pi(L)_0=L_1, \Pi(L)_1=L_0$ and $x*v=(-1)^{p(x)}xv,x\in\mathfrak{gl}(m,n).$ For every $L\in \mathcal F$ we can define 
$$
ch\,L=\sum_{\chi}\dim L_{\chi}e^{\chi}
$$
where the sum is taken over all weights of $L$. It is easy to see that $ch\,L$ is well defined function on $K(\mathcal F)$.

The ring $K(\mathcal F)$ can be describe explicitly in the following way.  Let 
$$
P_{m,n}=\Bbb Z[x_1^{\pm1},\dots, x_m^{\pm1},\, y_1^{\pm1},\dots, y_n^{\pm1}]
$$
be the ring of Laurent polynomials  in variables $x_1,\dots,x_m$ and $y_1,\dots, y_n.$ 

If we set $x_i=e^{\varepsilon_i},\, y_j=e^{\delta_j}$ then we get a character map 
$$
ch : K(\mathcal F)\longrightarrow P_{m,n}.
$$
 Let   also 
$$
\Lambda^{\pm}_{m,n}=\{f\in P_{m,n}^{S_m\times S_n}\mid x_i\frac{\partial f}{\partial x_i}+y_j\frac{\partial f}{\partial y_j}\in(x_i+y_j) \}
$$
be the subring of $P_{m.n}$ of supersymmetric Laurent  polynomials. 
\begin{thm}\cite{SV} The   ring  $K(\mathcal F)$ is isomorphic to the ring $\Lambda^{\pm}_{m,n}$ under the character map.
\end{thm}
For any parabolic subalgebra $\frak{p}\subset \frak{gl}(m,n)$ and any finite dimensional representation of $M$ of $\frak p$ by a super version of Borel--Weil--Bott construction one can define the corresponding Euler character $ch\,E^{\frak p}(M)$. According to the general formula due to Serganova \cite{Serga}
$$
ch\,E^{\frak p}(M)=\sum_{w\in W_0}w\left(\frac{De^{\rho}ch\,M}{\prod_{\alpha\in R_{\frak{p}}\cap R_1^+}(1-e^{-\alpha})}\right)
$$
with
$$
D=\frac{\prod_{\alpha\in R_1^+}(e^{\alpha/2}+e^{-\alpha/2})}{\prod_{\alpha\in R_0^+}(e^{\alpha/2}-e^{-\alpha/2})}
$$
Here $\rho$ is the half-sum of even positive roots minus the half-sum of odd positive roots, $R_{\frak p}$ is the set of roots $\alpha$ such that $\frak{gl}(m,n)_{\pm\alpha}\subset\frak{p}$.  Let  $(r,s)$ be  a pair of integers  such that $0\le r \le m,\,$ $ 0\le s\le n,\,\,r-s=m-n$.   the following  Consider the system of simple roots
$$
\{\varepsilon_i-\varepsilon_{i+1},\delta_j-\delta_{j+1}, \varepsilon_r-\delta_1,\delta_s-\varepsilon_{r+1},
\varepsilon_m-\delta_{s+1}\},\,i\in[1,m]\setminus\{r\},\,j\in[1,n]\setminus\{s\}.
$$
So we have the corresponding set of positive even and odd roots. 
 Consider  now the parabolic subalgebra  $\frak p$ with
$$
R_{\frak p}=\{\varepsilon_i-\varepsilon_j,\,\delta_p-\delta_q,\, \pm (\varepsilon_i-\delta_p)\},
$$
where  $r+1\le i,j\le m,\,i\ne j$ and $s+1\le p,q\le n,\,p\ne q$.

If we set 
$$
\chi_{r,s}=\sum_{i=1}^r\lambda_i\varepsilon_i+\sum_{j=1}^s\mu_j\delta_j
$$
where 
$$
\lambda=(\lambda_1,\dots,\lambda_r),\,\,\mu=(\mu_1,\dots,\mu_s)
$$
are non increasing sequences of integers
then $\chi$ defines one dimensional representation of $\frak p$. Then it is easy to check that

$$
ch\,E(\chi_{r,s})\Delta(x)\Delta(y)=
$$

\begin{equation}
=\left\{\prod_{(i,j)\in D_{+}}\left(1+\frac{y_j}{x_i}\right)\prod_{(i,j)\in D_{-}}\left(1+\frac{x_i}{y_j}\right)x_1^{\lambda_1}\dots x_r^{\lambda_r}y_1^{\mu_1}\dots y_s^{\mu_s}x^{\rho_m}y^{\rho_n}\right\}
\end{equation}

where 
$$
\Delta(x)=\prod_{i<j}(x_i-x_j),\quad \Delta(y)=\prod_{p<q}(y_p-y_q)
$$
$$
D_+=\{(i,j)\mid 1\le i\le r,\, 1\le j\le n\},\,
$$
$$
 D_{-}=\{(i,j)\mid r+1\le i\le m,\, 1\le j\le s\}
$$
$
 x^{\rho_m}=x_1^{m-1}\dots x_m^{0},\,\, y^{\rho_n}=y_1^{n-1}\dots y_n^{0}.
$
and $r-s=m-n$.

  Let us denote by $P(m,n)$ the set of pairs of sequences of non-increasing  integers  $(\lambda,\mu)$ such that $\lambda=(\lambda_1,\dots,\lambda_r),\,\,\mu=(\mu_1,\dots,\mu_s)$
$$
 r\le m,\, s \le n,\,\,  r- s=m-n
$$
We will  also denote  $E(\chi_{r,s})$ by $E_{\lambda,\mu}$.
The following Theorem shows that Euler characters forms a basis in the ring $K(\frak{gl}(m,n))$.
\begin{thm} \label{basis}(\,\cite{Ser}) If   $(\,\lambda,\mu)\in P(m,n)$ then  the set of all $
ch\,E_{\lambda,\mu}\,
$
forms a basis in the ring  $\Lambda^{\pm}_{m,n}$.
\end{thm}

\begin{remark}
If $r=m,\,s=n $, then $E_{\lambda,\mu}$ is  a Kac module. If $\lambda$ is  a partition such that $\lambda_{m+1}\le n$ and $SP_{\lambda}(x,y)$ is  a Schur  superfunction then the following equality holds true
$$
SP_{\lambda}(x,-y)=ch\,E_{\tau,\sigma}(x,y),\quad \tau=(\lambda_1,\dots,\lambda_r),\,\,\sigma=(\lambda_1'-r,\dots,\lambda_s'-r)
$$
where $r=i(\lambda),s=j(\lambda)$ and $-y=(-y_1,\dots,-y_n)$(see \cite{SV2}).
 \end{remark}
The following corollary is one of the most important property of $E_{\lambda,\mu}$.
\begin{corollary}\label{stab} Let $\varphi: \Lambda_{m,n}^{\pm}\rightarrow \Lambda_{m-1,n-1}$ be the canonical homomorphism $\varphi(x_m)=-\varphi(y_n)$. Then 

$1)$
$
\varphi(E_{\lambda,\mu}(x_1,\dots,x_m,y_1,\dots,y_n))=E_{\lambda,\mu}(x_1,\dots,x_{m-1},y_1,\dots,y_{n-1})
$

$2)$  $\varphi(E(\lambda,\mu))=0$ if $(p,q)=(m,n)$.
\end{corollary} 
\begin{proof}Follows from the Jacoby-Trudy type formula (\cite{Ser}) and explicit formula for Kac modules.
\end{proof}

\section{Weight and cap diagrams}

We will use the language of the weight  and cap diagrams due to Brundan and Stroppel  \cite{BS}  but we will use it here in the  form due to I. Musson and V. Serganova \cite{MS}.
Let $(\lambda,\mu)\in P(m,n)$,
$
\lambda=(\lambda_1,\dots,\lambda_p),\,\,\mu=(\mu_1,\dots,\mu_q)
$.
 Then we can define two finite sets
$$
 A_{\lambda}=\{\lambda_1,\lambda_2-1,\dots,\lambda_p+1-p\},\,
 $$
 $$
 B_{\mu}=\{q-p-\mu_q,q-1-p-\mu_{q-1},\dots,1-p-\mu_1\}
$$
\begin{definition}\label{weight}Let $(A,B)$ be a pair of finite  subsets in $\Bbb Z$.  Then the  corresponding weight  diagram is the following function on $\Bbb Z$
$$
f(x)=\begin{cases} \times,\,\,x\in A\cap B\\
\circ,\,\,x\notin A\cup B\\
>,\,\,x\in A\setminus B\\
<,\,\, x\in B\setminus A
\end{cases}.
$$
We see from the definition that
$$
A=f^{-1}(>,\times),\quad B=f^{-1}(<,\times).
$$
For every diagram $f$ let us  define the following numbers:  $m(f)=|f^{-1}(\times,>)|$; $n(f)=|f^{-1}(\times,<)|$; super-dimention $d(f)=m(f)-n(f)$.
\end{definition}

 We also define  for  every weight diagram  a sign. 
 Let $X,Y\subset\Bbb Z$ and $X\cap Y=\emptyset$.  Let us set
$$
\delta(X,Y)=\sum_{x\in X,\,y\in Y}\delta(x,y),\quad \delta(x,y)=\begin{cases} 1,\text{if}\, x<y\\
0,\text{otherwise}
\end{cases}.
$$
\begin{definition}
For every  weight diagram $f$ we set $\varepsilon(f)=(-1)^{\theta(f)}$, where 
$$
\theta(f)=S(f^{-1}(\times))-\delta(f^{-1}\{>,<\},f^{-1}(\times))+
\frac12|f^{-1}(\times)|(|f^{-1}(\times)|-1)
$$
$$
+|f^{-1}(<)||f^{-1}(\times)| 
$$
and $S(f^{-1}(\times))$ means the sum of all element in $f^{-1}(\times)$.
\end{definition}
 
\begin{proposition}  The correspondence  
$
(\lambda,\mu)\rightarrow (A_{\lambda}, B_{\mu})
$
gives a bijection between $P(m,n)$ and the set $\mathcal E(m,n)$  consisting of all  weight diagrams $f$ such that $m(f)\le m,\,n(f)\le n, d(f)=m-n$.
\end{proposition}
\begin{proof} Direct verification.
\end{proof}
So for any $f\in\mathcal E(m,n)$ we will denote by $E(f)$ the corresponding Euler character. If $m(f)=m,\,n(f)=n$ then $E(f)=K(f)$ is the character of Kac module.

\begin{definition} Let $h$ be a diagram and $C\subset h^{-1}(\circ)$. Then by $h*C$ we will denote  the following diagram
$$
(h*C)^{-1}(x)=\begin{cases}
h^{-1}(x),\,\text{if}\,x=<,>\\
h^{-1}(x)\cup C,\, \text{if} \,x=\times\\
h^{-1}(x)\,\setminus C,\,\text{if}\, x=\circ
\end{cases}.
$$
\end{definition}

 The following Theorem is one of the main results of the paper  \cite{Ser3}. It uses the canonical bilinear form on the representation ring.
\begin{thm} \label{KacEuler} Let $K(f)$ be a Kac module and $E(p)$ be the Euler module then
\begin{equation}\label{main1}
(ch\,K(f),ch\, E(p))=\begin{cases}(-1)^{\theta(f)+\theta(p)}\, \text{if}\,f=p*C,\,\,\, \,C\subset p^{-1}({\circ})\cap\Bbb Z_{\le -d(p)}\\
0,\,\,\text{otherwise}
\end{cases}\end{equation}.
 \end{thm}

We also need  the notion of   the cap diagram \cite{BS}.
\begin{definition} Let $f$ be a weight diagram.  The corresponding cap diagram can be obtained in the following way. Take the rightmost  $\times$ and make the cap by joining it to the first $\circ$ on the right. Then take the next $\times $ to the left  and make the cap by joining it  to the first $\circ$ on the right which is not the end of a cap, and so on. 

We can also define the cap diagram in another way. Let $f(a)=\times$. Let us take the first $c$ on the right such that $f(c)=\circ$ and  the numbers of $\times$-s and $\circ$-s in the interval $(a,c)$  are the same. Then draw the cap joining $a$ and $c$. 
\end{definition}

If $f(a)=\times$ we will denote  by $C_a$ the cap with the initial point $a$. The final point of the cap $C_a$ will be denote by $\tilde a$. 

 In order to define the rising operators   we need the notion of a numbered weight  diagram.  A numeration  of $f$ is a  bijection $\xi: \mathcal [1,r]\rightarrow f^{-1}(\times)$ where $r=| f^{-1}(\times) |$.  We also have a partial  order on the $f^{-1}(\times)$ : $a\dashv a'$ means that $C_a\supset C_{a'}$ (in other words $a\le a'$ and $\tilde a'\le \tilde a$). Using this partial order we can define connected component of the weight diagram.  If weight diagram is connected   then it has the only minimal element.

\begin{definition}
 We will call a numeration  $\xi$  admissible  if  it satisfies the following condition : if $\xi(i)\dashv\xi(j)$ then $i<j$. Every weight diagram has a unique admissible numeration when $\xi$ is a monotonic function. In other words $i<j$ implies that $\xi(i)<\xi(j)$. We will call it the standard numeration.
 \end{definition}

 The rising operators $R_i$ were  define by Brundan  (\cite{B}).

\begin{definition} Let  $f$  be  a numbered  weight  diagram  and $r=|f^{-1}(\times)|$.  If $i\in [1,r]$ then  we can define the rising operator $R_i$  by the following rule:  $R_i(f)$ is  the weight diagram  obtaining  from $f$ by moving the cross $\xi(i)$ to the end of the cap with the initial  point $\xi(i)$. 
\end{definition}

\begin{definition} Let $f$ be a weight diagram with standard numeration and $g$ be any diagram. Let us denote  by $X(g,f)$  the set of sequences $(i_1,\dots,i_r)\in\Bbb Z^r$ such that $g=R_1^{i_1}\dots R_r^{i_r}(f)$.
\end{definition}

\section{Creation operators and Su-Zhang formula}
In this section we give one more prove a the formula  for the decomposition of irreducible   module in terms of Kac modules \cite{SZ,Ser2}.

The next definition gives the most useful tool to deal with  the rising operators. 

\begin{definition}\label{create} Let $f$ be a  weight diagram.   Let us define  creation operators $P_a(f)$ for $a\in \Bbb Z\setminus f^{-1}\{>,<\}$ by the following rules:

$1)$ $ (P_a f)^{-1}(>)=f^{-1}(>),\,\, ( P_a f)^{-1}(<)=f^{-1}(<);\,\,$

$2)$ if $b< a$ then $ (P_af)(b)=f(b)$;

$3)$  Set $M=\{b\in f^{-1}(\times)\mid b\ge a\}$. The corresponding new positions  of the elements of $M$ can be define by the following rule. Take the rightmost  element  $b$ from $M$ and   move it to  the second  $\circ $ on the right.  Then do  the same for the  rightmost  element in   $M\setminus\{b\}$ and  so on. 

$4)$  $(P_af)(a)=\times$. 
\end{definition}

\begin{remark}  We can describe the above definition  in geometric terms. All symbols $<,>$ remain unchanged.  We add one cross to position $a$. All symbols $\times$  strictly on the left of $a$ remain unchanged. All symbols  $\times$ of diagram $f$ on the right of $a$ 
move  to the second empty position beginning  from the rightmost one. 
\end{remark}

For any diagram $f$ we will denote by $f_{core}$  the diagram obtaining from $f$ by replacing all crosses by symbols $\circ$.

\begin{example}
Consider the following  cap diagram $f$
$$
\xymatrix{\times&>&> &\circ &\times &<& \circ}
$$
We want to construct it by means of creation operators.
We start with the following diagram $f_{core}$
$$
\xymatrix{\circ&>&> &\circ &\circ &<& \circ}.
$$
Then we  see that $P_1(f_{core})$ is the following diagram
$$
\xymatrix{\times &>&>&\circ &\circ &<& \circ}. 
$$
 In the same way   $P_1P_1(f_{core})$ we have 
$$
\xymatrix{\times&>&>&\circ &\times &<& \circ}
$$
Therefore  $f=P_1P_1(f_{core})$.

It is  also easy to check  that  $f=P_5 P_1(f_{core})$. 
\begin{remark} 

Actually it is easy to check that  any diagram can be represented in the form
$$
f=P_{b_r}\dots P_{b_1}(f_{core})
$$
where $f^{-1}(\times)=\{b_1<\dots,b_r\}$.
\end{remark}
\end{example}
The following Lemma allows to express the action of the operators  $R_i$ in terms of creation operators.

\begin{lemma}\label{main2}   Let  $f=P_{b_r}\dots P_{b_1}(f_{core})$ and $f^{-1}(\times)=\{b_1<\dots<b_r\}$. Then  the following equality holds true
\begin{equation}\label{mainq}
R_1^{i_1}\dots R_r^{i_r}(P_{b_r}\dots P_{b_1}(f_{core}))=P_{a_r}\dots P_{a_1}(f_{core})
\end{equation}
where $a_p-n(a_p)=b_p-n(b_p)+i_p,\,p=1,\dots,r.$
\end{lemma}
\begin{proof}
Let $r=1$ and $f^{-1}(\times)=\{b\}$. Then  it is easy to check that  $R_1^{i}(P_b(f_{core}))=P_a(f_{core})$ where
$
a= b +n(a)-n(b)+i.
$
 Therefore in this case the formula is true. If $r>1$ then we  set $res\, f= P_{b_{r-1}}\dots P_{b_1}(f_{core})$. Since  $f$ has the standard numeration  and according to the case $r=1$  we have $R_{r}^{i_r}(f)=P_{a_r}(res\,f)$. By induction we have   
$$
R_1^{i_1}\dots R_{r-1}^{i_{r-1}}(res\,f)= P_{a_{r-1}}\dots P_{a_1}(f_{core})
$$
Therefore we need to prove the equality$$
R_1^{i_1}\dots R_{r-1}^{i_{r-1}}R_r^{i_r}(f)= P_{a_r}R_1^{i_1}\dots R^{i_{r-1}}_{r-1}(res\,f)
$$
When we apply all other rising operators  to $R_{r}^{i_r}(f)$ then these  operators have to jump  over the cap with number $r$. But this is the same as to apply operator  $P_{a_r}$ to the function  $R_1^{i_1}\dots R_{r-1}^{i_{r-1}}(res\,f)$. Therefore we proved equality (\ref{mainq}).
\end{proof}

\begin{definition} Let $a=(a_1,\dots,a_r)$ be a sequence of integers.
We will denote by $Y(a)$ the set of  integer sequences  $y=(y_1,\dots,y_r)$  such that : 

$1)$ $ y_1\le\dots\le y_r$ ;

$2)$   $ y_i\le a_i$.

\end{definition}

\begin{corollary}\label{cor}  Let $f,g$ be a weight diagrams such that  $f^{-1}\{<,>\}=g^{-1}\{<,>\}$. Then there  exists a bijection between $X(g,f)$ 
and sequences $(a_1,\dots,a_r)$ such that  $g=P_{a_r}\dots P_{a_1}(f_{core})$ and   $b\in Y(a).$

\end{corollary}
\begin{proof} 
First let us prove that if $a,b\notin f^{-1}(<,>)$ then $b<a$ is equivalent to $b-n(b)<a-n(a)$. Indeed if $a>b$ then $a-b>n(a)-n(b)$ and $a-n(a)>b-n(b)$. If  $a-n(a)>b-n(b)$ and $b>a$ then $b-n(b)>a-n(a)$. So we came to contradiction.

Now  let  $(i_1,\dots,i_n)$  be such that $g=P_1^{i_1}\dots P_r^{i_r}(f)$. Then by   Lemma \ref{main2}  we get the sequence   $(a_1,\dots,a_r)$ such that
 $$
 a_p-n(a_p)=b_p-n(b_p)+i_p,\,p=1,\dots,r
 $$
 Therefore $a_p-n(p)-(b_p-n(b_p))=i_p\ge 0$. Therefore $a_p\ge b_p$.

 So we proved that conditions $1),2)$ are nessesary. Let us prove that they are sufficient. We only need to indicate the values of $(i_1,\dots,i_r)$. Let $b_p\le a_p$. Then  $a_p-n(a_p)-(b_p-n(b_p))=i_p\ge0$. Corollary is proved.
\end{proof}

By  the above Corollary in  order  to find the  number of solutions of the equation  $g=R_1^{i_1}\dots R_r^{i_r}(f)$  we need to find the number of sequences $(a_1,\dots,a_r)$ satisfying the conditions of  Corollary \ref{cor}. But the first conditions   $g=P_{a_r}\dots P_{a_1}(f_{core})$ is the most difficult to verify. So we replace it by another equivalent condition which is more easy to verify. First note that every presentation $g=P_{a_r}\dots P_{a_1}(g_{core})$  defines the numeration on the weight diagram $g$.
\begin{remark} It is easy to see that in core free case the following relation between creations operators holds true
$$
\text{if}\, a\le b\, \text{then}\, P_aP_b=P_{b+2}P_a
$$
 Using this relation it is clear that for any sequence of integers $a=(a_1,\dots,a_r)$  the product $P_{a_1}\dots P_{a_r}$ can be rewritten in the 
 $P_{b_r}\dots P_{b_1}$ where $b_r>\dots>b_1$. This gives the following rule to  calculate the number $|X(g,f)|$. Denote the sequence $b_1,\dots,b_r$ as $m(a)$.

 Let us write  $f$ be a diagram. We also consider it as the sequence $f_1<\dots<f_r$ and the same for diagram $g=(g_1<\dots<g_r)$. Then the number $|X(g,f)|$ is equal to the number of integer sequences  $a$ such that  $m(a)=g$ and $f_i\le a_i,\,i=1,\dots,r$.
\end{remark}
 \begin{example} Let $f=(531-1),\,g=(6321)$. In the sequence $(6321)$ we can only permute the last two symbols and we  get $(3421)$. In the last sequence we can permute $(4,2)$ and we get $(3221)$. Therefore we have three sequences
 $$
 6321,\quad 3421,\quad 3221
 $$
 and we have the only sequence $(6321)\succ (531-1)$.
 \end{example}
\begin{definition} Let $g=P_{a_r}\dots P_{a_1}(g_{core})$.  Then let us  define the numeration $\xi$ on the weight diagram $g$ by the following rule: $\xi(i)$ is the final position of the cross which was added by  the operator $P_{a_i}$.
\end{definition}

It turns out that we can replace  the sequence  $a_1,\dots, a_n$  by the corresponding numeration $\xi$.

\begin{lemma}  Let $g=P_{a_n}\dots P_{a_1}(g_{core})$ and $\xi$ be the corresponding numeration of $g$. Then for diagram $f=P_{b_r}\dots P_{b_1}(f_{core})$ with the standard numeration   relation $b\in Y(a)$ is equivalent to relation $b\in Y(\xi)$ where $\xi=(\xi(1),\dots,\xi(n)).$

\end{lemma}
\begin{proof} Let $b\in Y(a)$ and  $b=(b_1\le \dots\le  b_r)$.  Then  we  have  $b_i\le a_i\le \xi(i)$. Therefore $b\in Y(\xi)$.

Now let us prove the opposite statement. Let $b\in Y(\xi)$ then  we need to  prove that $b_i\le a_i$. Suppose that there exists $i$ such that $b_i\le \xi(i),\,i=1,\dots,r$ and $b_i>a_i$ for some $i$. We can suppose that $i$ is the maximal with such a property. Note that $i\ne r$ since $a_r=\xi(r)$. Therefore  for any $j>i$ we have  $a_i<b_i<b_j\le a_j$. Therefore $a_i<a_j$ for any $j>i$  and by definition of $\xi$ we see that $\xi(i)=a_i<b_i$. So we came to contradiction. Therefore $b\in Y(a)$ and we proved the Lemma.
\end{proof}

So we see in order to find $|X(g,f)|$ we need to describe  the numerations $\xi$ which correspond  the relation $g=P_{a_r}\dots P_{a_1}(g_{core})$.

Let us recall that a  numeration $\xi$ of a  weight diagram $g$  is called admissible if  it preserves the partial order. In other words   if cap with initial point at $\xi(i)$ is located above the cap with  initial point  at $\xi(j)$  then $i<j$.

\begin{lemma}
  Conditions $g=P_{a_r}\dots P_{a_1}(f_{core})$ and $b\in Y(a)$ are equivalent to the conditions $\xi$ is admissible and $b\in Y(\xi).$ 
\end{lemma}

\begin{proof}  We only need to prove that conditions  $g=P_{a_r}\dots P_{a_1}(f_{core})$ and $\xi$ is admissible are equivalent.  Let us prove  first that the  corresponding numeration  $\xi$ is admissible. The cross with number $i$ appears as  the result of $P_{a_i}\dots P_{a_1}(f_{core})$  and the corresponding cap  does not contained any previous caps. Therefore relation $j<i$ implies  that  the cap with the initial point at $\xi(i)$ does not contain  the cap with the initial point at $\xi(j)$. Therefore if $\xi(i)\dashv \xi(j)$ then $i<j$.

Now let $\xi$ be an admissible numeration of a weigh diagram $g$. Let us prove by induction that there exist the sequence $a_1,\dots,a_r$ 
such that  $g=P_{a_r}\dots P_{a_1}(g_{core})$  with the admissible numeration $\xi$. By definition of $P_{a_r}$ we have $a_r=\xi(r)$. Since $\xi$ is admissible we see that if  integer $c\ne a_r,\tilde a_r$  and  it is located   under the cap $C_{a_r}$ then $c\in g^{-1}(>,<)$.  Define the  weight diagram  $h$ by the following rule: 

$1)$ $h^{-1}(<)=g^{-1}(<),\,\,h^{-1}(>)=g^{-1}(>)$;

$2)$ if $b<a_r$ then $h^{-1}(b)=g^{-1}(b)$;

$3)$   replace $a_r$ by $\circ$. Set $M=\{b\in g^{-1}(\times)\mid b> a_r\}$. The corresponding new positions  of the elements of $M$ can be define by the following rule. Take the leftmost  element  $b$ from $M$ and   move it to  the second  $\circ $ on the left.  Then do  the same for the  leftmost element in   $M\setminus\{b\}$ and  so on. 

It is easy to verify  that $P_{a_r}(h)=g.$ Besides under inductive assumption we have $h=P_{a_{r-1}}\dots P_{a_1}(g_{core})$. Therefore 
$g=P_{a_r}\dots P_{a_1}(g_{core})$ and again by induction for given $\xi$ there is only one sequence $a=(a_1,\dots,a_r)$.
\end{proof}

\begin{corollary} Let $f$ be a weight diagram with the standard numeration and $g$ be a diagram without numeration. Then $|X(g,f)|$ is equal to the number of pairs $(\xi,f)$ where $\xi$ is admissible numerations  of $g$  and   $f\in Y(\xi).$ We will denote the set of these pairs by $Z(g,f)$.
\end{corollary}

\begin{example}\label{adm0} 
 Let 
 $$
 f^{-1}(\times)=\{-1,1,3,5\}, \,f^{-1}\{<,>\}=\emptyset,\,g^{-1}\{<,>\}=\emptyset,\,\,g^{-1}(\times)=\{1,2,3,6\}.
 $$ 
 Then it is easy to see that there are  three  admissible numerations $\xi_1, \xi_2,\xi_3$ of $g$.

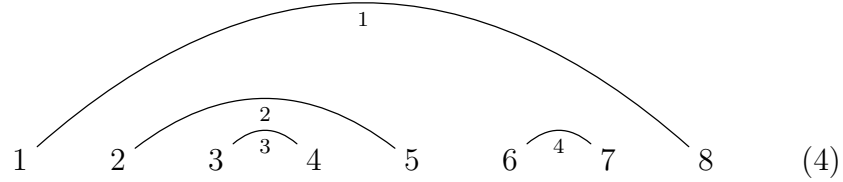
\begin{equation}\label{adm1}
\xymatrix{{1}\ar@/^5pc/@{}[rrrrrrr]_1&{2}\ar@/^2pc/@{}[rrr]_{2} &3\ar@/^1pc/@{}[r]_3&4 &5& 6\ar@/^1pc/@{}[r]_4&7&8}
\end{equation}
\vskip0.5cm
From the picture  above we see  that $\xi_1=(1,2,3,6)$. Therefore we have
$$
-1\le 1,\,\,\,1\le2,\, \,\,3\le 3,\,\,\,  5\le 6,\quad  f\in Y(\xi_1)
$$

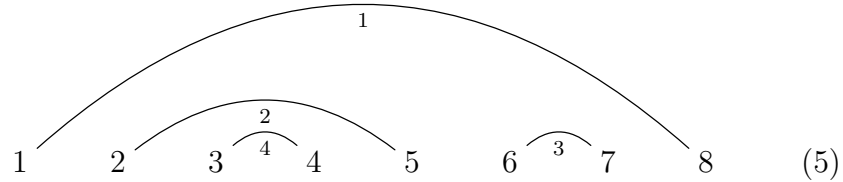
\begin{equation}\label{adm2}
\xymatrix{{1}\ar@/^5pc/@{}[rrrrrrr]_1&{2}\ar@/^2pc/@{}[rrr]_{2} &3\ar@/^1pc/@{}[r]_4&4 &5& 6\ar@/^1pc/@{}[r]_3&7&8}
\end{equation}
\vskip0.5cm
From the picture  above we see  that $\xi_2=(1,2,6,3)$
$$
-1\le 1,\,\,\,1\le2,\, \,\,3\le 6,\,\,\,  5> 3,\quad  f\notin Y(\xi_2)
$$

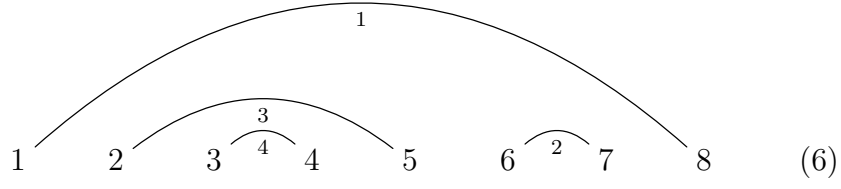
\begin{equation}\label{adm3}
\xymatrix{{1}\ar@/^5pc/@{}[rrrrrrr]_1&{2}\ar@/^2pc/@{}[rrr]_{3} &3\ar@/^1pc/@{}[r]_4&4 &5& 6\ar@/^1pc/@{}[r]_2&7&8}
\end{equation}
\vskip0.5cm
From the picture  above we see  that $\xi_3=(1,6,2,3)$. Therefore we have
$$
-1\le 1,\,\,\,1\le6,\, \,\,3> 2,\,\,\,  5> 3,\quad  f\notin Y(\xi_3)
$$ 
Therefore  $|Z(f,g)|=1$.
\end{example}
\begin{definition} Let $f,g$ be diagrams 
and $g^{-1}(\times)=\{a_1<\dots<a_r\}$  Let us denote by $D(g,f)\subset \Bbb Z_{\ge0}^r$ the set of $x=(x_1,\dots,x_r)$ such that

$1)$ $x_i\in f^{-1}(\times),\,i=1,\dots,r $;

$2)$ if $a_i\dashv a_j$ then $x_i\le x_j$;

$3)$ $x_i\le a_i$

\end{definition} 
\begin{proposition}\label{biject} Consider the map $\Phi  :Z(g,f)\rightarrow \Bbb Z^r$ such that 
$$
\Phi(\xi)=(x_1,\dots,x_r),\quad x_i=\xi_{st}\xi^{-1}(a_i)
$$
 Then   $\Phi$ is a bijection   $Z(g,f)$ on to $D(g,f)$.  
\end{proposition}
\begin{proof}  Let us prove first that  $\Phi(\xi)\in D(g,f)$ for admissible $\xi$. It is easy to see the $\Phi(\xi)$ can be rewritten in the following form: if $\xi(p)=a_i$ then $x_i=f_p$. Therefore condition $x_i\in f^{-1}(\times)$ is clear. Suppose that $a_i\dashv a_j$  therefore 
$\xi(p)=a_i\dashv a_j=\xi(q)$. Therefore $p<q$  and we have $x_i=f_p\le f_q=x_j.$  We also have $x_i=f_p\le \xi(p))=a_i$.

Now let $x\in D(g,f)$. We are going to prove that there exists $\xi$ such that $f\in Y(\xi)$. 
 Let $1\le p\le r$ then there is only one $i$ such that $x_i=f_p$. Then we set $\xi(p)=a_i$. Let $a_i=\xi(p)\dashv\xi(q)=a_j$.  Then $f_p=x_i<x_j=f_q$. Therefore $p<q$.

\end{proof}
\begin{corollary}   Let $H(g,f)$ be the set of functions $\varphi: g^{-1}(\times)\rightarrow f^{-1}(\times)$ such that :

$1)$ $\varphi$ is a bijection;

$2)$ for any $a\in f^{-1}(\times)$ we have $\varphi(a)\le a$;

$3)$ if $a_1\dashv a_2$ then $\varphi(a_1)<\varphi(a_2).$ 

Then $H(g,f)$  is in bijection with $D(g,f)$ under correspondence  such that $\varphi(a_i)=x_i$.
\end{corollary}

 Let $V$ be a vector space with the basis $K( f,\xi)$ where $(f,\xi)$ run over  all  numbered weight diagrams.  We also can define the action of rising operators on $V$ by the following formulae
 $$
R_i(K(f,\xi))= K((R_i(f,\xi)))
 $$
 Let also $U$ be a vector space with the basis $K(f)$ where $f$ run over all weight diagrams. We have a natural map $F: V\rightarrow U$  such that $F(K(f,\xi))=K(f)$.
 
 \begin{lemma}\label{prod}  Let $(g,\xi)$ be a numbered weight  diagram  with admissible  numeration $\xi$ and  $f^{-1}(\times)=\{\xi(i_1)<\dots<\xi(i_r)\}$.  Then   the following equality holds true
$$
F((1+R_{r})\dots(1+R_{1})K(g,\xi))=F((1+R_{i_r})\dots(1+R_{i_1})K(g,\xi))
$$

\end{lemma}
\begin{proof}

Let  $(g_1,\xi_1),\dots,(g_s,\xi_s)$ be the connected  components of $g$.   It is easy to see that $R_{\alpha},\,R_{\beta}$ pairwise commute  if $\alpha,\beta$ belong to different  $g_i,g_j$. Therefore it is enough  to prove Lemma for the case when $g$ is connected. If it is so  then $\xi(i_1)\dashv \xi(i)$ for any $i\ne i_1$. By our assumption on $\xi$ we have $i_1<i$ for any $i\ne i_1$. Therefore $i_1=1$. Let $\tilde g=g\setminus\{a_1\}$ and $\tilde \xi$ is the restriction of $g$ on to $\{2,\dots,s\}$. By induction Lemma is true for $(\tilde g,\tilde\xi)$ and therefore  it is true for $(g,\xi).$
\end{proof}

\begin{corollary}  The following equality holds true
$$
chL(g)=\sum_{f}(-1)^{\theta(f)+\theta(g)}d_{g,f}chK(f),\,\,d_{g,f}=|D(g,f)|
$$
\end{corollary}
\begin{proof}  We have the following equality 
$$
K(f,\xi_{st})=\sum_{i_1.\dots,i_r}(-1)^{i_1+\dots+i_r}(1+R_r)\dots (1+R_1)R_1^{i_1}\dots R_r^{i_r}K(f,\xi_{st})
$$
$$
=\sum_{(g,\xi)=R_1^{i_1}\dots R_r^{i_r}(f,\xi_{st})}(-1)^{i_1+\dots+i_r}(1+R_r)\dots (1+R_1)K(g,\xi)
$$
and it is not difficult to check that $i_1+\dots+i_r\equiv\theta(f)+\theta(g) \mod 2$ if $(g,\xi)=R_1^{i_1}\dots R_r^{i_r}(f,\xi_{st}).$

Let us apply to the both side of the equality the operator $F$. Therefore  by Lemma  \ref{prod}  and Theorem  \ref{biject} we have 
$$
K(f)=\sum_{g}(-1)^{\theta(f)+\theta(g)}d_{gf}P(g)
$$
Since families  $\{L(g)\},\,\{K(f)\}$ are dual to the families $\{P(g)\},\,\{K(f)\}$ with respect to  the canonical bilinear form  then we have 
$$
chL(g)=\sum_{g}(-1)^{\theta(f)+\theta(g)}d_{g,f}chK(f)
$$
\end{proof}
 
 \section{Irreducible modules and Euler characters}
 The following formula gives the decomposition of Euler character in terms of Kac modules. 
 \begin{corollary}\label{EulerKac} Let $p$ be a weight diagram. Then the following equality holds true
 $$
 E(p)=\sum_{f=p*C}(-1)^{\theta(p)+\theta(f)}K(f),\quad C\subset\Bbb Z_{\le -d(p)}
 $$
 \end{corollary}  
 
 \begin{proof} Let us  denote  by $N$ the right hand side of the above equality and let $K(f)$ be a Kac module. It is enough to prove that $(K(f),E(p))=(K(f),N)$ for any $f$. If $f$ does not has the form $f=p*C$ for any $C\supset \Bbb Z_{\le -d(p)}$ then we have  by  Theorem \ref{KacEuler}  $(K(f),E(p))=0=(K(f),N)$ since Kac modules pairwise orthogonal. If $f=p*C$  for some $C\supset \Bbb Z_{\le -d(p)}$ then again by the same Theorem we have $(K(f),E(p))=(-1)^{\theta(f)+\theta(p)}=(K(f),N)$.
 \end{proof}

 Let $p,g\in \mathcal E(m.n)$ and $m(g)=m,\,n(g)=n$. Suppose that for some $C\subset \Bbb Z_{\le -d(g)}$ we have  $d_{g,p*C}\ne0$.  Then it is easy to see that  if $\max C<\min g^{-1}(\times)$ then $d_{g,p*C}$ does not depend on $C$. In such a case  we will  set $d^{\,\infty}_{g,p}=d_{g,p*C}$.

\begin{thm}\label{Euler} Let 
$$
L(g)=\sum_{p}(-1)^{\theta(f)+\theta(p)}e_{g,p}E(p)
$$
be the decomposition  irreducible module $L(g)$ into  Euler characters. Then the following equality holds true
$$
e_{g,p}=\sum_{q*Z=p}(-1)^{|Z|}d^{\,\infty}_{g,q},\quad Z\subset \Bbb Z_{\le -d(g)}
$$

\end{thm}
\begin{proof}
From Theorem \ref{basis} it follows that any irreducible module can be represented as the finite sum of Euler characters. 
$$
L(g)=\sum_{p}(-1)^{\theta(f)+\theta(p)}e_{g,p}E(p)
$$
 Let us decompose  every Euler character in the above decomposition as the sum Kac modules according to Corollary \ref{EulerKac}
$$
E(p)=\sum_{f=p*C}(-1)^{\theta(p)+\theta(f)}K(f),\quad C\subset\Bbb Z_{\le -d(f)}
$$
Therefore we have 
$$
L(g)=\sum_{p}(-1)^{\theta(g)+\theta(p)}e_{g,p}\sum_{f}(-1)^{\theta(p)+\theta(f)}K(f)=
$$
$$
=\sum_{f}(-1)^{\theta(g)+\theta(f)}\left(\sum_{p}e_{g,p}\right)K(f)
$$
Therefore  we have 
$$
d_{g,f}=\sum_{p*C=f}e_{g,p}.
$$
Let  us take $f=q*X$. If we take $X\subset \Bbb Z_{\le -d(g)}$ such that $x<y$ for any $x\in X$ and any $y\in  g^{-1}(\times)\bigcup_p p^{-1}(\times)$. Then $q=p*Z$ for some $Z\subset \Bbb Z_{\le -d(g)}$.  Therefore we have 
$$
d^{\,\infty}_{g,q}=\sum_{p*Z=q}e_{g,p}.
$$
If we apply the inversion formula then we get
$$
e_{g,p}=\sum_{q*Z=p}(-1)^{|Z|}d^{\,\infty}_{g,q}
$$
\end{proof}
\begin{proposition}\label{irreducible} $E(p)$ is irreducible  if and only if $p^{-1}(\times)$ is empty. 
\end{proposition}
\begin{proof} Let $E(p)$ be irreducible and  
$$
E(p)=\sum_{C\subset \Bbb Z_{\le d}}(-1)^{\theta(p)+\theta(p*C)}K(p*C)
$$
be its  decomposition into Kac modules. All sets $C$ have the same number  of elements  $r$. Let us  represent any $C$ as $C=\{c_1>\dots>c_r\} $. Let $c^0_1$ be the maximal element from all $\{c_1\}$ and the same for $c_2^0,\dots, c_r^0$. Let $g=p*\{c^0_1,\dots,c^0_r\}.$  So we see that the highest term of  $E(p) $ is $K(g)$. Therefore  decompositions of  $E(p)$ and $ L(g)$  coincide up to sign. But in such a case  for all $f$  in the decomposition of $L(g)$ we have $f^{-1}(\times)\supset p^{-1}(\times)$. This is only possible when $p^{-1}(\times)=\emptyset$.

Let us prove the opposite statement.  Let $p^{-1}(\times)=\emptyset$ and $g$ be as above. Therefore the order $\dashv$ on $g^{-1}(\times)$ is the same as natural order. Therefore  all multiplicities  in the decomposition of $L(g)$ are equal to one. Because of $p^{-1}(\times)=\emptyset$  we see that  the decomposition $L(g)$ in terms of Kac modules coincides  with  the decomposition $E(p)$. Therefore $E(p)$ is irreducible.
\end{proof}
\begin{proposition} Let 
$$
L(g)=\sum_{p}(-1)^{\theta(g)+\theta(p)}e_{g,p}E(p)
$$ 
be the decomposition of irreducible module in terms of Euler modules.  Let $q=g_{core}$   Then the following equalities hold true
$$
sdim\, L(g)=(-1)^{\theta(g)+\theta(q)}e_{g,q}\,sdim\, E(q)
$$
$$
=(-1)^{\theta(g)+\theta(q)}d^{\infty}_{g,q}\,sdim\, L(q)=
$$

\end{proposition}
\begin{proof} Let $\varphi : \Lambda^{\pm}_{m,n}\rightarrow \Lambda^{\pm}_{m-1,n-1} $  be a homomorphism  such that $\varphi(x_m)=-\varphi(y_n)$ and $\psi=\underbrace{\varphi\circ\dots\circ\varphi}_{k}$ where $k=|f^{-1}(\times)|$. It is known that $\varphi$ preserves the super-dimension. Therefore 
$$
sdim\, L(g)=sdim\,\psi( L(g))
$$
But it follows from Corollary \ref{stab} that $\psi(E(p))=0$ if $p\ne q$. Since $q^{-1}(\times)=\emptyset$ we have $\psi(E(q))=E(q)=L(q)$. Moreover
from Theorem \ref{Euler}  we see that $e_{q,f} = d^{\infty}_{q,f}$ and proposition follows.
\end{proof}

Let us introduce  a  countable set $ \bar1<\bar2<\dots<\bar k<\dots$  and we suppose that $\bar k<x$ for any $x\in \Bbb Z$. Then the  numbers $d^{\,\infty}_{g,p}$ can be interpreted in the following way.
\begin{lemma}\label{infty}  Let $p^{-1}(>,<)=g^{-1}(>,<)$  then the  number  $d^{\,\infty}_{g,p}$  is equal to the number of maps $\varphi: g^{-1}(\times)\rightarrow p^{-1}(\times)\cup\{\ \bar1,\dots\bar k\}$ such that 

$1)$ $\varphi $ is a bijection

$2)$ $\varphi $ preserves  the orders

$3)$  if  $\varphi(a) \in p^{-1}(\times)$ then $\varphi(a)\le a$.
\end{lemma}

\begin{proof} A proof  easily follows from the definition of $d^{\,\infty}_{g,p}$.
\end{proof}
\begin{corollary}  (see \cite {HW})The following equality holds true 
$$
d^{\infty}_{g,g_{core}}=\frac{|g^{-1}(\times)|!}{\prod_{a\in g^{-1}(\times)}|T_a |},\,\, T_a=\{b\in g^{-1}(\times)\mid a\dashv b \}
$$

\end{corollary}
\begin{proof} It easily follows from the Lemma above and induction on $|g^{-1}(\times)|$.
\end{proof}
\begin{thm}  Let $\varphi : \Lambda^{\pm}_{m,n}\rightarrow \Lambda^{\pm}_{m-1,n-1} $ be a homomorphism  such that $\varphi(x_m)=-\varphi(y_n)$.
Then the following equality holds true
$$
\varphi(L(g))=\sum_{i=1}^r(-1)^{\theta(f)+\theta( g_i)}L(g_i)
$$
where $g_1,\dots g_r$ are the diagrams corresponding to  the connected components of $f^{-1}(\times)$  and  they  are obtained from $f$ by deleting the maximal  cap from the corresponding connected component.
\end{thm}
\begin{proof} Let $f_i$ be the same as in the formulation of the theorem.   Let us prove first that  if $|q^{-1}(\times)|<|f^{-1}(\times)|$ then
$$
d^{\,\infty}_{f,q}=\sum_{i=1}^rd^{\,\infty}_{f_i,q}
$$
Since $|q^{-1}(\times)|<|f^{-1}(\times)|$ the number $k>0$ and $\bar k<a$ for any $a\in\Bbb Z$. Therefore if $\varphi(a)=\bar k$ then $a$ is a minimal one  and other elements from the corresponding connected component are mapped  to $p^{-1}(\times)\cup\{\bar1,\dots,\overline{q-1}\}$ and equality  follows. Now we have  
$$
e_{f,p}=\sum_{q*Z=p}(-1)^{|Z|}d^{\infty}_{f,q}.
$$
If  $|p^{-1}(\times)|<|f^{-1}(\times)|$ then $|q^{-1}(\times)|<|f^{-1}(\times)|$  and we have 
$$
e_{f,p}=\sum_{q*Z=p}(-1)^{|Z|}\sum_{i=1}^kd^{\,\infty}_{f_i,q}=\sum_{i=1}^k\sum_{q*Z=p}(-1)^{|Z|}d^{\,\infty}_{f_i,q}=\sum_{i=1}^ke_{f_i,p}
$$
If  $|p^{-1}(\times)|=|f^{-1}(\times)|$ then $e_{p,f}=d_{p,f}$ and $E(p)$ is a Kac module. Therefore $\varphi(E(p))=0$. So we see that
$$
\varphi(L(f))=\sum_{p}\sum_{i=1}^k(-1)^{\theta(f)+\theta(p)}e_{f_i,p}E(p)=\sum_{i=1}^k(-1)^{\theta(f)+\theta(f_i)}L(f_i).
$$
Therefore the theorem is proved.
\end{proof}
 It is not easy to calculate the  coefficients   $e_{f,p}$ in Theorem \ref{Euler}. But in some special  cases the corresponding sum can be reduced to one term. 
\begin{proposition}\label{euler1} Let $g$ be a weight diagram such that $g^{-1}(\times)\subset \Bbb Z_{>n-m}.$ Let us denote by $f$ the  highest weight diagram as in Proposition \ref{irreducible}.  Then the following equality holds true
$$
L(f)=\sum_{p^{-1}(\times)\subset \Bbb Z_{>n-m}}(-1)^{\theta(f)+\theta(p)}d^{\,\infty}_{f,p}E(p)
$$
\end{proposition}
\begin{proof} We need to prove that $e_{f,p}=d^{\,\infty}_{f,p}$. We know that 
$$
e_{f,p}=\sum_{q*Z=p}(-1)^{|Z|}d^{\,\infty}_{f,q}
$$
where  sum is taken over all $Z\subset p^{-1}(\times)\cap\Bbb Z_{\le n-m}$.

First let us prove that $p^{-1}(\times)\subset \Bbb Z_{>n-m}$. Suppose that this is not true and $c$ be the minimal (in ordinary sense) element in that intersection. Then we have 
$$
e_{,pf}=\sum_{q*Z=p,\,c\in Z}(-1)^{|Z|}(d^{\,\infty}_{q,f}-d^{\,\infty}_{q*\{c\},f})
$$
Therefore we see that we only need to prove that  $d^{\,\infty}_{f,q}=d^{\,\infty}_{f,q*\{c\}}$. Consider two sets
$$
X=q^{-1}(\times)\cup\{\bar1,\dots\bar k\},\quad Y=q^{-1}(\times)\cup\{c,\bar2,\dots\bar k\}
$$
Then it is easy to verify that function $\tau$  sending $\bar1$ to $c$ and leaving  all other elements unchanged is a monotonic bijection. 
Now let $\varphi: f^{-1}(\times)\rightarrow X$ be a bijection  with the properties as in Lemma \ref{infty}. Let us check that $\tau\circ\varphi$ has the same properties. Clearly $\tau\circ\varphi$  has property $2)$ since $\tau$ is monotonic. Let us prove that $\tau\circ\varphi(a)\le a$ for any $a\in f^{-1}(\times)$. If $\varphi(a)\ne \bar1$ then we have $\tau\circ\varphi(a)=\varphi(a)\le a$. Let $\varphi(a)=\bar1$ then  $\tau\circ\varphi(a)=c$ and we need to prove that $c\le a$. If $a\ge n-m$ then $c\le n-m\le a$. Let us suppose that $a<n-m$. We know that 
$f^{-1}(\times)=g^{-1}(\times)\cup\{c_1,\dots,c_l\}$ where $c_1>c_2>\dots>c_l$ all   belong to $\Bbb Z_{\le n-m}$. Since $a<n-m$ we see that $a=c_i$ for some $i$. Moreover since $q^{-1}(>,<)=f^{-1}(<,>)=p^{-1}(>,<)$ we have $c=c_j$ for some $j$ or $c<c_i,\,i=1,\dots,l$. In the second case we see that $c<a$. If $a=c_1$ then $c=c_j\le c_1=a$, Therefore we can suppose that $a=c_i,\, i\ge2$ and $c=c_j$. Let us show that $j>i$. Indeed if $j<i$ then $c_i\dashv c_j$ and $\varphi(c_j)>\varphi(c_i)=\varphi(a)=\bar1$. Therefore $\varphi(c_j)\in q^{-1}(\times)\subset p^{-1}(\times)$. Therefore $\varphi(c_j)>c_j$ since the minimality of $c$ and $c\notin q^{-1}(\times)$. So we got a contradiction. 
Therefore $\tau\circ\varphi$ satisfies all properties from the Lemma \ref{infty}.
\end{proof}
\begin{example} Let $f^{-1}(\times)=\{1,2,4\},f^{-1}(<)=f^{-1}(>)=\emptyset.$ Then we have
$$
ch\,L(1,2,4)=2ch\,E(\emptyset)-2ch\,E(1)+2ch\,E(2)-ch\,E(3)+ch\,E(4)+
$$
$$
2ch\,E(1,2)-ch\,E(1,3)+
$$
$$
+ch\,E(1,4)+ch\,E(2,3)-ch\,E(2,4)-ch\,E(1,2,3)+ch\,E(1,2,4)
$$
Therefore $sdim\, L(1,2,4)=2$. 
\end{example}

 Let $g$ be such that $g^{-1}(\times)\subset\Bbb Z_{>n-m}$. Then we can give a geometric proof of the decomposition of irreducible module in terms of Euler characters. Let us define the polyhedron   $\Delta_g$  by the following conditions:

$1)$  $x_i=c_i$ if $g(c_i)=<,>$; 

$2$ $n-m\le x_i\le c_i$ if $g(c_i)=\times$;

$3)$ $x_i\le x_j$ if $g(c_i)=g(c_j)=\times$  and $c_i\dashv c_j.$ 

For  $x\in \Delta_{g}$  let us set $N_x=\{c_i\in g^{-1}(\times)\mid x_i=n-m\}$. The set $N_x$  is ordered   by means $\dashv$.  as a subset of $g^{-1}(\times)$. Let   $c_{g,x} $  be  the number of  monotonic bijections from $N_x$ to $\{\bar1,\dots,\bar s\}$ where $s=|N_x|$.

\begin{proposition}  

 The  following equality holds true
$$
ch\,L(g)=\sum_{x\in \Delta_g}c_{g,x}ch\,E(x)
$$
\end{proposition}
\begin{proof} Let $D_g$ be the polyhedron defined by the conditions  (see \cite{Ser}):

$1)$  $x_i=c_i$ if $g(c_i)=<,>$; 

$2$ $x_i\le c_i$ if $g(c_i)=\times$;

$3)$ $x_i\le x_j$ if $g(c_i)=g(c_j)=\times$  and $c_i\dashv c_j.$

Let us define the function $\chi  :\Bbb R\rightarrow \Bbb R$
$$
\chi(\alpha)=\begin{cases} 0,\,\, \alpha\le n-m\\
\alpha,\,\, \alpha >n-m
\end{cases},\quad \chi(\alpha)=\frac12(sgn(\alpha-n-m)\alpha+\alpha)
$$
Then we can define the map $F:\Bbb R^N\rightarrow \Bbb R^{N}$ such that: if $x=(x_1,\dots,x_N)\in \Bbb R^N$ then
$ F(x)=(F_1(x),\dots,F_N(x)) $ where
$$
F_i(x)=\begin{cases}x_i,\,\text{if}\, g(c_i)=<,>\\
\chi(x_i),\, \text{if}\, g(c_i)=\times
\end{cases}
$$
 It is not difficult to verify that $F(D_g)=\Delta_g$.
 Therefore  by \cite{Ser}
 $$
 L(g)=\sum_{y\in D_g}K(g)=\sum_{x\in\Delta_g}\sum_{y\in F^{-1}(x)}K(y).
 $$
 But 
 $$
 \sum_{y\in F^{-1}(x)}K(y)=c_{g,x}E(x)
 $$
 It is easy to verify  that $y\in F^{-1}(x)$  if and only if $y_i=x_i$ if $x_i>n-m$ and  for all other $y_j$  we have $y_j\le n-m$.  Let us set $C_x=\{c_i\in g^{-1}(\times)\mid x_i\le n-m\}$. The set $C_x$  is ordered   by means $\dashv$.  as a subset of $g^{-1}(\times)$ and $c_{g,x} $ is equal the number of  monotonic bijections from $C_x$ to $\{\bar1,\dots,\bar s\}$ where $s=|C_x|$.

\end{proof}

{\bf Acknowledgement.}
This work is done at the Center of Pure Mathematics MIPT. It is financially supported by
Russian Science Foundation grant 26-11-00115.

{\bf Conflict of Interest}
 The author has no conflict of interest to declare that are relevant to this article.


\begin{thebibliography}{99}
\bibitem{B}
Brundan, J.: Kazhdan-Lusztig polynomials and character formulae for Lie superalgebra $\frak{gl}(m|n).$
J. Amer. Math. Soc. 16(1), (2003)
\bibitem{BS} 
Brundan, J., Stroppel, C.: Highest weight categories arising from Khovanov's diagram algebra IV:
the general linear supergroup. J. Eur. Math. Soc. 14, (2012)
\bibitem{D}
P. Deligne, Categories tensorielles, Mosc. Math. J. 2 (2) (2002) 227--248.
\bibitem{GS}
Gruson, C., Serganova, V.: Cohomology of generalized supergrassmannians and character formulae for basic classical Lie superalgebras. Proc. Lond. Math. Soc. 101(3), (2010)
\bibitem{GH}
M. Gorelik, Th. Heidersdors:
Gruson- Serganova character  formulas and the Duflo-Serganova  cohomology functor J. reine angew. Math. 798 (2023), 1--54 Journal fur die reine und angewandte Mathematik DOI 10.1515/crelle-2022-0080
\bibitem{HW}
Heidersdorf,T. Weissauer,R.: Cohomological Tensor Functors on Representations of the General Linear Supergroup. Memoirs of the American Mathematical Society. March 2021, Volume 270,  Number 1320 (fourth of 7 numbers)
\bibitem{MS}
Musson, I., Serganova, V.: Combinatorics and characters formulae for Lie superalgebra gl(m, n). Transformation Groups 16, (2011)
\bibitem{Serga}
V. Serganova,:Characters of irreducible representations of simple Lie superalgebras, in Proceedings of the Int. Congress of Math., vol. II, Berlin 1998, Doc. Math. Extra Vol.II (1998) 583--593.
\bibitem{SV}
Sergeev, A.N., Veselov, A.P.: Grothendieck rings of basic classical Lie superalgebras. Ann. of Math. 173, (2011).
\bibitem{SV2}
 A.N. Sergeev., A.P. Veselov. : Euler characters and super Jacobi polynomials.  Adv. in Math. 228, (2011), 4286--4315.
\bibitem{Ser}
A.N. Sergeev.: On rings of supersymmetric polynomials. Journal of Algebra 517 (2019) 336--364.

\bibitem{Ser3}
A.N. Sergeev.: Canonical bilinear form and Euler characters. Journal of Algebra 601 (2022) 149--177.

\bibitem{Ser2}
Sergeev, A.: Combinatorics of Irreducible Characters for Lie Superalgebra $\frak{gl}(m,n).$
 Transformation Groups 31, 2021--2043 (2026). https://doi.org/10.1007/s00031-026-09961-3
 \bibitem{SZ}
  Su, Y., Zhang, R.B.: Character and dimension formulae for general linear superalgebra. Adv. Math.
211, (2007)
\end{thebibliography}
\end{document}